\documentclass[12pt,reqno]{amsart}
\usepackage{a4wide}
\usepackage[T1]{fontenc}
\usepackage{amssymb,amsmath,amsthm,mathtools}
\usepackage{mathrsfs} 
\usepackage{enumitem}
\usepackage[dvipsnames]{xcolor}
\usepackage{graphicx}
\usepackage{hyperref}
\usepackage{cleveref} 
\hypersetup{colorlinks=true, linkcolor=blue, citecolor=teal, urlcolor=cyan}
\usepackage{comment}
\usepackage{amsrefs} 

\newtheorem{theorem}{Theorem}[section]

\newtheorem*{claim*}{Claim}

\newcounter{maintheorem}

\theoremstyle{remark}

\theoremstyle{definition}

\newtheorem{example}[theorem]{Example}
\numberwithin{equation}{section}
\makeatother

\newcommand{\R}{\mathbb{R}}
\newcommand{\ZZ}{\mathbb{Z}}
\newcommand{\N}{\mathbb{N}}
\newcommand{\e}{\varepsilon}

\newcommand{\nn}[1]{{\left\vert\kern-0.25ex\left\vert\kern-0.25ex\left\vert #1 
\right\vert\kern-0.25ex\right\vert\kern-0.25ex\right\vert}}

\renewcommand{\leq}{\leqslant}
\renewcommand{\geq}{\geqslant}
\renewcommand{\tilde}{\widetilde}

\newcommand{\di}{\mathsf{d}}
\newcommand{\de}{\mathrm{d}}

\DeclareMathOperator{\lip}{lip}
\DeclareMathOperator{\Lip}{Lip}

\DeclareMathOperator{\pen}{pen}

\usepackage[colorinlistoftodos,textsize=tiny]{todonotes}

\usepackage{bclogo}

\title[Global Lipschitz extension preserving the slope]{Global Lipschitz extension preserving the slope}

\author[N.~De Ponti]{Nicolò De Ponti}
\address[N.~De Ponti]{Politecnico di Milano, Dipartimento di Matematica, Piazza Leonardo da Vinci 32, 20133 Milano, Italy \newline
\href{https://orcid.org/0000-0001-7951-5987}{\texttt{ORCID:0000-0001-7951-5987}}}
\email{nicolo.deponti@polimi.it}

\author[J.~Somaglia]{Jacopo Somaglia}
\address[J.~Somaglia]{Politecnico di Milano, Dipartimento di Matematica, Piazza Leonardo da Vinci 32, 20133 Milano, Italy \newline
\href{https://orcid.org/0000-0003-0320-3025}{\texttt{ORCID:0000-0003-0320-3025}}}
\email{jacopo.somaglia@polimi.it}


\begin{document}
\begin{abstract}
We show that every real-valued Lipschitz function on a subset of a metric space can be extended to the whole space while preserving the slope and, up to a small error, the global Lipschitz constant. This answers a question posed by Di Marino, Gigli, and Pratelli, who established the analogous property for the asymptotic Lipschitz constant. We also prove the same result for the ascending slope and for the descending slope.
\end{abstract}
\maketitle

\section{Introduction}
Given a metric space $(X,\di)$, a subset $C\subset X$, and a $L$-Lipschitz function $g:C\to \R$, a classical result due to McShane \cite{McShane} (see also \cite{W34}) establishes that the function
\begin{equation}\label{eq: McS}
f:X\to \R, \qquad f(y)\coloneqq \inf_{x\in C} \{g(x)+L\di(x,y)\}    
\end{equation}
is a $L$-Lipschitz function that extends $g$ to the whole $X$. 

Beyond preserving the global Lipschitzianity, a natural question to explore is whether the extension is possible while also preserving other properties of the function $g$. We refer to \cites{CG,DeBeVe,DGP,Gu25,Milman98}, and references therein, for different results in this direction. In particular, a recent result of Di Marino, Gigli, and Pratelli \cite{DGP} ensures that this is the case if one is interested in preserving the \textit{asymptotic Lipschitz constant} defined as 
$$ \lip_a(f,x)=\begin{cases} 0 &\mbox{ if } x \mbox{ is isolated, }\\  
\displaystyle \limsup_{y,z\to x}\frac{|f(y)-f(z)|}{\di(y,z)} &\mbox{ otherwise. }
\end{cases}
$$
More precisely, they proved that for every $\e>0$ a $L$-Lipschitz function $g:C\to\R$ can always be extended to a $(L+\e)$-Lipschitz function $f:X\to \R$ such that $\lip_a(g,x)=\lip_a(f,x)$ for every $x\in C$, and they noticed that the $\e$-dependence of the global Lipschitz constant cannot in general be removed. In the same article \cite{DGP}*{Remark c)}, the authors ask if the result remains true by replacing the asymptotic Lipschitz constant with the \textit{slope} (also known as \textit{local Lipschitz constant}) defined as
$$ \lip(f,x)=\begin{cases} 0 &\mbox{ if } x \mbox{ is isolated, }\\  
\displaystyle \limsup_{y\to x}\frac{|f(y)-f(x)|}{\di(y,x)} &\mbox{ otherwise. }
\end{cases}
$$

The aim of the paper is to answer the question of Di Marino, Gigli, and Pratelli by proving the following Theorem:

\begin{theorem}\label{t: mainthm}
    Let $(X,\di)$ be a metric space, $C\subset X$ a subset, and $g\colon C\to\R$ a $L$-Lipschitz function. Then, for every $\e>0$ there exists a $(L+\e)$-Lipschitz function $f\colon X\to \R$ whose restriction to $C$ coincides with $g$ and such that 
    \begin{equation*}
        \lip(g,x)=\lip(f,x)\,\,\,\, \mbox{ for every } x\in C.
    \end{equation*}
    Moreover if $g$ is bounded (resp. with bounded support), then $f$ can be chosen to be bounded (resp. with bounded support).
\end{theorem}

We remark that the $\e$-dependence is essential also in our situation, see Example \ref{ex: e-dependence}. The construction of the desired function $f$ recalls the one performed in \cite{DGP} and corresponds to a \emph{non-linear} version of the McShane's extension \eqref{eq: McS}. That is, we put  
\begin{equation}\label{eq: f in intro}
    f(y)\coloneqq \inf_{x\in C} \{g(x) + \pen_x ( \di(x,y)) \}
\end{equation}
where $\pen_x:[0,\infty)\to [0,\infty)$ is a suitable \textit{penalization} function whose derivative close to $0$ approximates the value $\lip(g,x)$ (see the first part of the proof for the precise definition). By closely reasoning as in \cite{DGP}, in the first three steps of the proof of our main result we show that the above function $f$ is indeed a $(L+\e)$-Lipschitz extension of $g$, while the fourth step proves that the infimum in \eqref{eq: f in intro} can be “localized” in a precise sense. The last two steps are, instead, entirely new and deviate from the arguments of \cite{DGP}, allowing us to prove that $f$ preserves the slope. For every $x\in C$, the idea is to bound from below all the functions $y\mapsto g(x) + \pen_x ( \di(x,y))$ with some suitable “piecewise-linear” functions, defined in \eqref{def: psi}, with a controlled slope. In the final step, we then show that this control is sufficient to provide a valid estimate of the incremental ratio $|f(x)-f(y)|/\di(x,y)$, $x\in C, y\in X$, concluding the proof.

\medskip

In order to have a quite complete picture of the extension theorems preserving the various concepts of slope, we also obtain the analogous result of Theorem \ref{t: mainthm} for the \textit{ascending slope} $\lip^+$ and the \textit{descending slope} $\lip^-$. For a function $f:X\to \R$ and a point $x\in X$, these quantities are defined as
\begin{equation*} \lip^{\pm}(f,x)=\begin{cases} 0 &\mbox{ if } x \mbox{ is isolated, }\\  
\displaystyle \limsup_{y\to x}\frac{(f(y)-f(x))_\pm}{\di(x,y)} &\mbox{ otherwise,}
\end{cases}
\end{equation*}
where by $h_+$ (resp. $h_-$) we denote the positive part (resp. the negative part) of a function $h\colon X\to \R$.

\begin{theorem}\label{t: +- slope}
    Let $(X,\di)$ be a metric space, $C\subset X$ a subset, and $g\colon C\to\R$ a $L$-Lipschitz function. Then, for every $\e>0$ there exists a $(L+\e)$-Lipschitz function $f\colon X\to \R$ whose restriction to $C$ coincides with $g$ and such that 
    \begin{equation*}
        \lip^{\pm}(g,x)=\lip^{\pm}(f,x)\,\,\,\, \mbox{ for every } x\in C.
    \end{equation*}
    Moreover if $g$ is bounded (resp. with bounded support), then $f$ can be chosen to be bounded (resp. with bounded support).
\end{theorem}

Here, the extension function $f$ is constructed similarly to \eqref{eq: f in intro}, with a natural change in the penalization term. The proof then follows almost straightforwardly, without relying on the refined estimate required for Theorem \ref{t: mainthm}.

\medskip 

The notion of (ascending/descending) slope is involved in the definition of several quantities in the theory of metric spaces and gradient flows \cites{Cheeger,AGS}, and Theorems \ref{t: mainthm} and \ref{t: +- slope} serve as a technical tool in various possible situations where one is interested in extending functions while controlling the local differential behavior. To cite an application in the context of metric Sobolev spaces, we mention that a possible way (see, e.g., \cite{AGScalc}) to characterize $H^{1,2}$-functions over a metric measure space $(X,\di,\mathfrak{m})$ is by considering $L^2(X,\mathfrak{m})$ functions with finite Cheeger energy, the latter being defined as the $L^2$-relaxation of the pre-Cheeger energy
\[\mathsf{pCh}(f)\coloneqq\int_X\lip^2(f,x)\de\mathfrak{m},\,\qquad f\in \Lip_{bs}(X,\di),\]
$\Lip_{bs}(X,\di)$ being the space of Lipschitz functions $f:X\to \R$ with bounded support. 
Our main result can then be used to show, in a direct and simple way, that the resulting Sobolev space is invariant under isomorphisms of metric measure spaces, without relying on its non-trivial equivalence with the definition based on test plans, as discussed in \cite{AGScalc, AGSiber}. We refer to \cite{DGP}*{Theorem 3.1} for all the details (see also \cite{KLR}*{Corollary 2.4} for the case of $\mathrm{BV}$-functions), where the corresponding result is proved for the Sobolev space defined starting from the integration of the asymptotic Lipschitz constant.

\section{Preliminaries and applications}

Let $(X,\di)$ be a metric space. We will denote by $B_r(x)$ the open ball of radius $r>0$ and center $x\in X$. For a nonempty set $C\subset X$, a function $g\colon C\to \R$, a point $x\in C$ and a set $A\subset C$, we denote by

\begin{equation*}
\begin{split}
    \Lip(g,A,x)&\coloneqq\inf\{\ell\geq 0\colon |g(z)-g(x)|\leq \ell \di(x,z) \ \ \forall z\in A\},\\
\Lip^{\pm}(g,A,x)&\coloneqq\inf\{\ell\geq 0\colon (g(z)-g(x))_{\pm}\leq \ell \di(x,z) \ \ \forall z\in A\},
\end{split}
\end{equation*}
and 
\begin{equation*}
    \Lip(g,A):=\inf\{\ell\geq 0\colon |g(z_1)-g(z_2)|\leq \ell \di(z_1,z_2) \ \ \forall z_1,z_2\in A\}.
\end{equation*}

We recall the definitions of local Lipschitz constants we are interested in. The quantity
\begin{equation*}
    \lip_a(g,x):=\inf_{r>0}\Lip(g,C\cap B_r(x))=\lim_{r\to 0}\Lip (g,C\cap B_r(x))
\end{equation*}
will be called \textit{asymptotic Lipschitz constant} of $g$ at $x$, while
\begin{equation*}
    \lip(g,x):=\inf_{r>0}\Lip(g,C\cap B_r(x),x),
\end{equation*}
will be called \textit{slope} of $g$ at $x$, and
\begin{equation*}
    \lip^{+}(g,x):=\inf_{r>0}\Lip^{+}(g,C\cap B_r(x),x), \quad  \lip^{-}(g,x):=\inf_{r>0}\Lip^{-}(g,C\cap B_r(x),x)
\end{equation*}
are the \textit{ascending slope} and the \textit{descending slope} of $g$ at $x$, respectively.
Clearly, all these quantities depend on the domain of the function $g$. 

Notice that if $x\in C$ is not an isolated point, it holds that
\begin{equation*}
    \lip(g,x)=\limsup_{y\to x}\frac{|f(y)-f(x)|}{\di(y,x)}, \qquad \lip^{\pm}(g,x)=\limsup_{y\to x}\frac{(f(y)-f(x))_\pm}{\di(x,y)}.
\end{equation*}

The next example shows that the asymptotic Lipschitz constant and the slope do not always coincide.
\begin{example}\label{ex: lip differt to lipa}
    Let $X=\{0\}\cup\bigcup_{n\in\N}[\frac{1}{n},\frac{1}{n}+\frac{1}{n^2}]$, $\di=|\cdot|$ be the standard Euclidean distance and $g\colon X\to \R$ defined by
    \begin{equation*}
        g(x)=    
    \begin{cases}
        0 &\mbox{ if } x=0,\\
        x - \frac{1}{n} &\mbox{ if } x\in [\frac{1}{n},\frac{1}{n}+\frac{1}{n^2}].
    \end{cases}
    \end{equation*}
    We are going to show that $\lip_a(g,0)>\lip(g,0)$. Indeed, let $x\in [\frac{1}{n},\frac{1}{n}+\frac{1}{n^2}]$, then
    \begin{equation*}
        \frac{|g(x)-g(0)|}{|x|}\leq \frac{g(\frac{1}{n}+\frac{1}{n^2})-g(0)}{\frac{1}{n}+\frac{1}{n^2}}=\frac{1}{n+1}.
    \end{equation*}
    Therefore, we have $\lip(g,0)=0$. On the other hand, we have
    \begin{equation*}
  \textstyle\Lip(g,[\frac{1}{n},\frac{1}{n}+\frac{1}{n^2}]) \geq \frac{g(\frac{1}{n}+\frac{1}{n^2})-g(\frac{1}{n})}{\frac{1}{n}+\frac{1}{n^2}-\frac{1}{n}}=1,
    \end{equation*}
    which implies $\lip_a(g,0)\geq 1>0=\lip(g,0)$.
\end{example}
For every Lipschitz function $g:X\to \R$ and every $x\in X$, it is possible to show that 
\[\lip_a(g,x)\geq \lip^{*}(g,x)\geq \lip(g,x) \geq \lip^{\pm}(g,x) \]
and the first inequality is an equality on a length metric space (see \cite{ACD}*{Proposition 12} for a proof). Here $\lip^{*}(g,x)$ is the upper semicontinuous envelope of the slope of $g$.

The next example shows that the dependence on $\e$ in Theorems \ref{t: mainthm} and \ref{t: +- slope} cannot be dropped.

\begin{example}\label{ex: e-dependence}
    Let $(X,\di)=([-1,2],|\cdot|)$, $C=[-1,0]\cup [1,2]$ and $g\colon C\to \R$ defined by 
    \begin{equation*}
        g(x)=    
    \begin{cases}
        0 &\mbox{ if } x\in [-1,0],\\
        1 &\mbox{ if } x\in [1,2].
    \end{cases}
    \end{equation*}
The function $g$ is $1$-Lipschitz and $\lip(g,x)=\lip^+(g,x)=\lip^-(g,x)=0$, for every $x\in C$. Moreover, the unique $1$-Lipschitz extension to the whole $X$ does not preserve the slope and the ascending slope at $x=0$, and the descending slope at $x=1$.
\end{example}

\section{Proof of the main results}\label{sec: proof}
In order to preserve the slope at each point of $C$, we slightly modify the construction of Di Marino, Gigli, and Pratelli in \cite{DGP}. Even though Steps 1-4 of the proof are essentially the same as those in \cite{DGP}, we have included them for completeness.
\begin{proof}[Proof of Theorem \ref{t: mainthm}]
Let $\e>0$ and without loss of generality we assume $L>0$ and $\e\leq L$. Let $\{\e_k\}_{k\in \ZZ}$ be a sequence such that

\begin{itemize}
\item $\e_k> 0$ for every $k\in\mathbb Z$; 
\item $k\mapsto\frac{\e_{k-1}}{\e_k} $ is increasing and $\frac{\e_{k-1}}{\e_k}\to 0$ as $k \to -\infty$;
\item for every $k \in \mathbb{Z}$ it holds
\begin{equation*}
\frac{ \e_{k-1} }{\e_k} \leq \frac \e{3(L+\e)}.
\end{equation*}
\end{itemize}

We consider, for $x\in C$, $S_k^c(x)\coloneqq \Lip(g,C\cap B_{\e_k}(x),x)$ and the penalization function $ \pen_x : [0,\infty) \to [0, \infty)$, defined as the only continuous function such that
\begin{equation*}
\pen_x(0)=0 \qquad \qquad \pen'_x(t) = S_k^c(x) + 3 L \frac{ \e_{k-2}}{\e_{k-1}} \qquad \text{ for }\e_{k-2} < t < \e_{k-1}.
\end{equation*}
Notice that the function $\pen_x$ is convex and Lipschitz, for each $x\in C$. We remark that \cite{DGP} uses the function $S_k(x)\coloneqq\Lip(g,C\cap B_{\e_k}(x))$ in place of $S_k^c(x)$.

Then we define:
\begin{equation*}\label{eqn:deff}
\begin{split}
\phi_x (y)&: = g(x) + \pen_x ( \di(x,y))\qquad\forall x\in C,\ y\in X, \\
 f(y)&:= \inf_{x \in C} \left\{ \phi_x(y) \right\}\qquad\qquad\qquad\qquad\ \ \forall  y\in X.
\end{split}
\end{equation*}

\noindent\textbf{Step 1.} We claim that
\begin{equation}
\label{eq:claim1}
\text{$\phi_x$ is $(L+\e)$-Lipschitz for every $x\in C$.}
\end{equation}
Indeed, the claim immediately follows by this simple estimate:
\[
\pen'_x(t) = S^c_k(x) + 3 L \frac{ \e_{k-2}}{\e_{k-1}} \leq L + 3 L  \frac {\e}{ 3(L+\e)} \leq L+ \e, \qquad \forall k\in\ZZ \mbox{ and } t \in (\e_{k-2}, \e_{k-1}).
\]

\vspace{4mm}

\noindent\textbf{Step 2.} We claim that
\begin{equation}
\label{eq:claim2}
\text{whenever $x,y \in C$ and  $\e_{k-1}\leq\di(x,y)\leq\e_{k}$, we have $\phi_x(y) \geq g(y) + \e_{k-2} L$. }
\end{equation}
In fact, using the definition of $S^c_k(x)$ it holds $g(x) \geq g(y)- S^c_k(x) \, \di(x,y)$, while $\pen_x(\di(x,y)) \geq \int_{\e_{k-2}}^{\di(x,y)} \pen'_x(t)\, \de t$ so that:
\begin{align*}
\phi_x(y)& =g(x) + \pen_x(\di(x,y)) \\
&\geq g(y)- S^c_k(x) \, \di(x,y)+ \int_{\e_{k-2}}^{\di(x,y)} \pen'_x(t)\, \de t  \\
\qquad\qquad&\geq  g(y) - S_k^c(x) \, \di(x,y) +  (\di(x,y) - \e_{k-2}) \, \big(S^c_k(x) + 3 L \frac{ \e_{k-2}}{\e_{k-1}}\big) \\
\qquad & \geq  g(y) - \e_{k-2} L + 3L (\e_{k-1} - \e_{k-2}) \, \frac{ \e_{k-2}}{\e_{k-1}}  \\
 &=  g(y) + \e_{k-2} L + L (\e_{k-1} - 3\e_{k-2}) \, \frac{ \e_{k-2}}{\e_{k-1}}\geq  g(y) + \e_{k-2} L,
\end{align*}
where we have used the properties of the sequence $\{\e_k\}_{k\in \ZZ}$.

\vspace{4mm}

\noindent \textbf{Step 3.} We claim that
\begin{equation}
\label{eq:claim3}
\text{ $f$ is an $(L+\e)$-Lipschitz extension of $g$.}
\end{equation}
Indeed, Step 2 implies that  $\phi_x(y) \geq g(y)$ for every $x,y \in C$, and thus $f(y) \geq g(y)$. On the other hand, $\phi_y(y)=g(y)$ for any $y\in C$, which implies $f(y) \leq g(y)$. Therefore $f$ is an extension of $g$. The fact that $f$ is an $(L+\e)$-Lipschitz function follows directly from \cite{DGP}*{Lemma 2.1} combined with Step 1.

\vspace{4mm}

\noindent \textbf{Step 4.}  We claim that
\begin{equation}
\label{eq:claim4}
\forall\bar x\in C,\ k\in\mathbb Z\quad\text{we have}\quad f(y)  = \inf_{x\in C \cap B_{\e_{k}}(\bar{x}) }   \phi_x(y)  \qquad \forall y \in B_{\e_{k-2}}(\bar{x}).
\end{equation}
To prove it, we will show that for $\bar x,k,y$ as above and $x\in C$ with $\di(x,\bar x)\geq \e_k$ it holds
\begin{equation}
\label{eq:claim44}
\phi_x(y)\geq f(y)+\e_{k-1}\frac L3.
\end{equation}
Since $f$ is $(L+\e)$-Lipschitz, we have
\[
f(y) \leq g(\bar{x})+ \e_{k-2} (L+\e).
\]
On the other hand, by the previous steps we also know that
$$\phi_x(y) \geq \phi_x(\bar{x}) - \e_{k-2}(L+\e) \geq g(\bar x) + \e_{k-1} L - \e_{k-2}(L+\e),$$
which allows to conclude. 

\vspace{4mm}

\noindent \textbf{Step 5.} Let $j\in \ZZ$ and $\bar{x},z\in C$ be such that $\e_{j-1}\leq\di(\bar{x},z)<\e_j$. We define 
\begin{equation}\label{def: psi}
     \psi_{\bar{x},z}:X\to \R,\qquad   \psi_{\bar{x},z}(y)\coloneqq\begin{cases}
            g(z) &\mbox{ if } \di(z,y)\leq \e_{j-2},\\
            g(z) + \frac{g(\bar{x})-g(z)}{\left(1-\frac{\e_{j-2}}{\di(\bar{x},z)}\right)\di(\bar{x},z)} (\di(y,z)-\e_{j-2})& \mbox{ otherwise.}
        \end{cases}
    \end{equation} 

We claim that,
    \begin{equation}\label{eq:claim5}
       \mbox{whenever } g(\bar{x})>g(z), \mbox{ we have}  \,\,\phi_z(y)\geq \psi_{\bar{x},z}(y) \quad \forall y\in X.
    \end{equation}
 We first observe that the values $\phi_z(y)$ and $\psi_{\bar{x},z}(y)$ depend only on the distance between $y$ and $z$. Then, it is enough to show that $\varphi_z(t)\geq \zeta_{\bar{x},z}(t)$ for every $t\in [0,\infty)$, where 
        \begin{equation*}
            \varphi_{z}(t)\coloneqq g(z) +\pen_z(t)
        \end{equation*}
    and 
        \begin{equation*}
        \zeta_{\bar{x},z}(t)\coloneqq\begin{cases}
            g(z) &\mbox{ if } t\leq \e_{j-2},\\
            g(z) + \frac{g(\bar{x})-g(z)}{\left(1-\frac{\e_{j-2}}{\di(\bar{x},z)}\right)\di(\bar{x},z)} (t-\e_{j-2})& \mbox{ otherwise }.
        \end{cases}
    \end{equation*}
    We observe that for any $y\in X$ we have $\phi_z(y)\geq g(z)$, so $\phi_z(y)\geq \zeta_{\bar{x},z}(y)$ for every $y$ with $\di(z,y)\leq \e_{j-2}$. Hence $\varphi_z(t)\geq \zeta_{\bar{x},z}(t)$ for $t\in [0,\e_{j-2}]$.\\
    To conclude the proof we are going to show that $\varphi'_z(t)>\zeta'_{\bar{x},z}(t)$ for every $t\in(\e_{k-2},\e_{k-1})$ and any index $\ZZ\ni k\geq j$. 
    We have 
    \begin{equation*}
        \varphi_z'(t)=\pen_z'(t)\geq S_k^c(z) +3L\frac{\e_{j-2}}{\e_{j-1}}\geq \frac{g(\bar{x})-g(z)}{\di(\bar{x},z)} + 3L\frac{\e_{j-2}}{\e_{j-1}}
    \end{equation*}
    while 
      \begin{equation*}
        \zeta_{\bar{x},z}'(t)=\frac{g(\bar{x})-g(z)}{(1-\frac{\e_{j-2}}{\di(\bar{x},z)})\di(\bar{x},z)}. 
    \end{equation*}
    It holds
    \begin{equation*}
    \begin{split}
  &\frac{g(\bar{x})-g(z)}{(1-\frac{\e_{j-2}}{\di(\bar{x},z)})\di(\bar{x},z)}-        \frac{g(\bar{x})-g(z)}{\di(\bar{x},z)} = \frac{g(\bar{x})-g(z)}{\di(\bar{x},z)} \left(\frac{1}{(1-\frac{\e_{j-2}}{\di(\bar{x}-z)})}-1\right)\\ &\leq L\left(\frac{\e_{j-2}}{\di(\bar{x},z)-\e_{j-2}}\right)
        \leq L \left(\frac{\e_{j-2}}{\e_{j-1}-\e_{j-2}}\right)= L \frac{e_{j-2}}{e_{j-1}}\frac{1}{1-\frac{e_{j-2}}{e_{j-1}}}< 3L \frac{e_{j-2}}{e_{j-1}},  
    \end{split}
    \end{equation*}
   where the last inequality follows from
   \begin{equation*}
       \frac{\e_{j-2}}{\e_{j-1}}\leq\frac{\e}{3(L+\e)}<\frac{2}{3}\quad \quad\mbox{ for every } j\in\ZZ.
   \end{equation*}
   So we obtain the desired inequality
       \begin{equation*}
    \begin{split}
  \zeta'_{\bar{x},z}(t)=\frac{g(\bar{x})-g(z)}{(1-\frac{\e_{j-2}}{\di(\bar{x},z)})\di(\bar{x},z)}<        \frac{g(\bar{x})-g(z)}{\di(\bar{x},z)} + 3L \frac{e_{j-2}}{e_{j-1}} \leq \varphi_{z}'(t), \quad \forall t\in(\e_{k-2},\e_{k-1}), k\geq j.
    \end{split}
    \end{equation*}

    \vspace{4mm}
    
\noindent \textbf{Step 6.}
We claim that
\begin{equation*}
    \lip(g,x)=\lip(f,x) \,\,\, \mbox{ for every } x\in C.
\end{equation*}

    \noindent
    Let $\bar{x}\in C$. If $\bar{x}$ is isolated in $X$ by definition $\lip(f,\bar{x})=\lip(g,\bar{x})=0$. Otherwise, let $\{y_n\}_{n\in\N}\subset X\setminus\{\bar{x}\}$ be such that $y_n\to \bar{x}$ as $n\to +\infty$ and 
    \begin{equation}\label{eq: lipfbiggerlipg}
        \lim_{n\to +\infty} \frac{|f(y_n)-f(\bar{x})|}{\di(y_n,\bar{x})}=\lip(f,\bar{x})\geq \lip(g,\bar{x}),
    \end{equation}
    where the inequality is an immediate consequence of the fact that $f$ extends $g$.
    Let $n\in\N$ and let $k_n\in \>Z$ be such that $e_{k_n-3}\leq\di(\bar{x},y_n)<\e_{k_n-2}$. 
    Notice that $k_n\to -\infty$ as $n\to +\infty$. We proceed by distinguishing two cases. Suppose that $f(\bar{x})\leq f(y_n)$, then, by the definition of the function $f$ and since $\pen_{\bar{x}}$ is a convex function with $\pen_{\bar{x}}(0)=0$, we have
    \begin{equation}\label{eq: upperbound}
    \begin{split}
            \frac{|f(y_n)-f(\bar{x})|}{\di(\bar{x},y_n)}&=\frac{f(y_n)-f(\bar{x})}{\di(\bar{x},y_n)}\leq \frac{\phi_{\bar{x}}(y_n)-f(\bar{x})}{\di(\bar{x},y_n)}\\
            &=\frac{\pen_{\bar{x}}(\di(\bar{x},y_n))}{\di(\bar{x},y_n)}\leq S_{k_n-1}^{c}(\bar{x}) +3L\frac{\e_{k_n-3}}{\e_{k_n-2}}.
    \end{split}
    \end{equation}
    Let us now suppose $f(y_n)<f(\bar{x})$. By \eqref{eq:claim4}, there exists $z_n\in C$ such that $\di(z_n,\bar{x})<\e_{k_n}$ and
    \begin{equation*}
        f(y_n)\geq \phi_{z_n}(y_n) - \eta_n\cdot\di(y_n,\bar{x}),
    \end{equation*}
    where $\eta_n>0$ is chosen in such a way that 
    \begin{itemize}
        \item $\eta_n<\frac{1}{n}$;
        \item $f(y_n)+\eta_n\di(y_n,\bar{x})<f(\bar{x})=g(\bar{x})$. 
    \end{itemize}
    Suppose that $\e_{j_n-1}\leq \di(\bar{x},z_n)<\e_{j_n}$ for some $j_n\leq k_n$, so that $j_n\to -\infty$ as $n\to+\infty$. Then we use \eqref{eq:claim5} (with $y=y_n$, $z=z_n$ and $j=j_n$), distinguishing two cases:
    if $\di(y_n,z_n)\leq \e_{j_n-2}$ we have
    \begin{equation}\label{eq: lowerbound1case}
    \begin{split}
            \frac{|f(y_n)-f(\bar{x})|}{\di(\bar{x},y_n)}&=\frac{f(\bar{x})-f(y_n)}{\di(\bar{x},y_n)}\leq \frac{f(\bar{x})-\phi_{z_n}(y_n)}{\di(\bar{x},y_n)}+\eta_n\\
            &\leq\frac{g(\bar{x})-g(z_n)}{\di(\bar{x},y_n)}+\eta_n\leq \frac{g(\bar{x})-g(z_n)}{\di(\bar{x},z_n)-\di(y_n,z_n)}+\eta_n\\
            &=\frac{g(\bar{x})-g(z_n)}{\di(\bar{x},z_n)}\frac{\di(\bar{x},z_n)}{\di(\bar{x},z_n)-\di(y_n,z_n)}+\eta_n\leq S_{j_n}^c(\bar{x})\frac{1}{1-\frac{\e_{j_n-2}}{\e_{j_n-1}}}+\eta_n;
    \end{split}
    \end{equation}
otherwise, in the case $\di(y_n,z_n)>\e_{j_n-2}$ it holds
\begin{equation}\label{eq: lowerbound2case}
    \begin{split}
            &\frac{|f(y_n)-f(\bar{x})|}{\di(\bar{x},y_n)}=\frac{f(\bar{x})-f(y_n)}{\di(\bar{x},y_n)}\leq \frac{f(\bar{x})-\phi_{z_n}(y_n)}{\di(\bar{x},y_n)}+\eta_n\\
    &\leq\frac{g(\bar{x})-g(z_n)}{\di(\bar{x},z_n)}\left[\frac{\di(\bar{x},z_n)}{\di(\bar{x},y_n)}-\frac{\di(y_n,z_n)-\e_{j_n-2}}{(1-\frac{\e_{j_n-2}}{\di(\bar{x},z_n)})\di(\bar{x},y_n)}\right]+\eta_n\\
    &=\frac{g(\bar{x})-g(z_n)}{\di(\bar{x},z_n)}\left[\frac{\di(\bar{x},z_n)-\di(y_n,z_n)}{(1-\frac{\e_{j_n-2}}{\di(\bar{x},z_n)})\di(\bar{x},y_n)}\right]+\eta_n\\
    &\leq S_{j_n}^c(\bar{x})\left[\frac{\di(\bar{x},y_n)}{(1-\frac{\e_{j_n-2}}{\di(\bar{x},z_n)})\di(\bar{x},y_n)}\right]+\eta_n
            \leq S_{j_n}^c(\bar{x})\frac{1}{1-\frac{\e_{j_n-2}}{\e_{j_n-1}}}+\eta_n
    \end{split}
    \end{equation}
Putting together \eqref{eq: upperbound}, \eqref{eq: lowerbound1case} and \eqref{eq: lowerbound2case}, we can infer
    \begin{align*}
        \lip(f,\bar{x})&=\lim_{n\to +\infty} \frac{|f(y_n)-f(\bar{x})|}{\di(y_n,\bar{x})}\\
          &\leq \lim_{n\to +\infty} \max\{S_{k_n-1}^{c}(\bar{x}) +3L\frac{\e_{k_n-3}}{\e_{k_n-2}},S_{j_n}^c(\bar{x})\frac{1}{1-\frac{\e_{j_n-2}}{\e_{j_n-1}}}+\eta_n\}= \lip(g,\bar{x}),
    \end{align*}
    which combined with \eqref{eq: lipfbiggerlipg} concludes the proof of the first part. 
    
    The part of the statement concerning bounded functions (and functions with bounded support) follows in the same way as in \cite{DGP}.
\end{proof}

Taking advantage of some of the arguments contained in the proof of Theorem \ref{t: mainthm}, we conclude the article proving Theorem \ref{t: +- slope}.

\begin{proof}[Proof of Theorem \ref{t: +- slope}]
Let us prove the result for the descending slope. 
    Let $\{\e_k\}_{k\in\ZZ}$ be a sequence of positive numbers defined in the same way as at the beginning of the proof of Theorem \ref{t: mainthm}. Let $x\in C$, we define $S^-_k(x)\coloneqq \Lip^-(g,C\cap B_{\e_k},x)$ and $ \pen_x : [0,\infty) \to [0, \infty)$, defined as the only continuous function such that
\begin{equation*}
\pen_x(0)=0 \qquad \qquad \pen'_x(t) = S_k^-(x) + 3 L \frac{ \e_{k-2}}{\e_{k-1}} \qquad \text{ for }\e_{k-2} < t < \e_{k-1}.
\end{equation*}
The function $\pen_x$ is convex and Lipschitz, for each $x\in C$. We define:
\begin{equation*}
\begin{split}
\phi_x (y)&: = g(x) - \pen_x ( \di(x,y))\qquad\forall x\in C,\ y\in X \\
 f(y)&:= \sup_{x \in C} \left\{ \phi_x(y) \right\}\qquad\qquad\qquad\qquad\ \ \forall  y\in X.
\end{split}
\end{equation*}
For each $x$, the function $\phi_x$ is $(L+\e)$-Lipschitz. It is enough to proceed as in the proof of Theorem \ref{t: mainthm} by observing that for each $x\in C$ and $k\in\ZZ$
\begin{equation*}
    S^-_k(x)\leq L,\,\,\,\mbox{and} \,\,\,3L\frac{\e_{k-2}}{\e_{k-1}}\leq \e.
\end{equation*}
Hence, by \cite{DGP}*{Lemma 2.1} the function $f$ is $(L+\e)$-Lipschitz. In order to prove that the function $f$ extends the function $g$, following the same idea of step 3 of the proof of Theorem \ref{t: mainthm} and since $\phi_y(y)=g(y)$ for every $y\in C$, it is enough to show that $\phi_x(y)\leq g(y)$ for each $x,y\in C$. To prove it, we fix $x,y\in C$ such that $\di(x,y)\in [\e_{k-1},\e_k]$ and notice that $g(x)\leq g(y) + S^-_k(x)\di(x,y)$. With analogous computations as in the proof of step 2 of Theorem \ref{t: mainthm}, we have
\begin{align*}
\phi_x(y)& =g(x) - \pen_x(\di(x,y)) \leq g(y)+ S^-_k(x) \, \di(x,y)- \int_{\e_{k-2}}^{\di(x,y)} \pen'_x(t)\, \de t  \\
 &\leq g(y) - \e_{k-2} L - L (\e_{k-1} - 3\e_{k-2}) \, \frac{ \e_{k-2}}{\e_{k-1}}\leq  g(y).
\end{align*}
Finally, it remains to show that 
\begin{equation*}
    \lip^-(g,x)=\lip^-(f,x),
\end{equation*}
for each $x\in C$. Let $\bar{x}\in C$ which is not isolated in $X$, otherwise $\lip^-(g,\bar{x})=\lip^-(f,\bar{x})=0$. Let $\{y_n\}_{n\in\N}\subset X\setminus\{\bar{x}\}$ such that $y_n\to\bar x$ as $n\to +\infty$ and 
\begin{equation*}
    \lim_{n\to +\infty} \frac{(f(y_n) - f(\bar x))_-}{\di(\bar x,y_n)}=\lip^-(f,\bar x) \geq \lip^-(g,\bar x).
\end{equation*}
For every $n\in \N$, if $f(\bar x)\geq f(y_n)$ and $\di(\bar x, y_n)\in [\e_{k_n-2},\e_{k_n-1})$, we have
\begin{equation*}
    \frac{(f(y_n) - f(\bar x))_-}{\di(\bar x, y_n)}=\frac{f(\bar x) - f(y_n)}{\di(\bar x, y_n)}\leq \frac{f(\bar x) - \phi_{\bar x}(y_n)}{\di(\bar x, y_n)}\leq \frac{\pen_{\bar x}(\di (\bar x,y_n))}{\di(\bar x,y_n)}\leq S_{k_n}^-(\bar x)+3L\frac{\e_{k_n-2}}{\e_{k_n-1}}.
\end{equation*}
Notice that the same inequality 
\[ \frac{(f(y_n) - f(\bar x))_-}{\di(\bar x, y_n)}\leq S_{k_n}^-(\bar x)+3L\frac{\e_{k_n-2}}{\e_{k_n-1}}\]
trivially holds if $f(\bar x)< f(y_n)$. By taking the limit on both sides, noticing that $k_n\to -\infty$ as $n\to +\infty$, we obtain the reverse inequality
\begin{equation*}
    \lim_{n\to +\infty}\frac{(f(y_n) - f(\bar x))_-}{\di(\bar x, y_n)}\leq \lip^-(g,x).
\end{equation*}
This concludes the proof for the descending slope. 

To extend $g$ while preserving the ascending slope, it is sufficient to take $f:X\to \R$ defined as $f\coloneqq -\tilde f$, where $\tilde f:X\to \R$ is an extension of $-g$ that preserves the descending slope. 

The second part of the statement follows in the same way as in \cite{DGP}.
\end{proof}

\paragraph{\em\bfseries Acknowledgments}
The research of the authors has been partially supported by the GNAMPA (INdAM -- Istituto Nazionale di Alta Matematica). NDP is supported by the INdAM-GNAMPA Project ``Proprietà qualitative e regolarizzanti di equazioni ellittiche e paraboliche'', codice CUP $\#E5324001950001\#$.



\begin{thebibliography}{WW}

\bibitem{ACD} L.~Ambrosio, M.~Colombo, and S.~Di Marino, \emph{Sobolev spaces in metric measure spaces: reflexivity and lower semicontinuity of slope}, Adv. Stud. Pure Math., \textbf{67}, Mathematical Society of Japan, [Tokyo], (2015), 1–58.

\bibitem{AGScalc} L.~Ambrosio, N.~Gigli, and G.~Savar\'e, \emph{Calculus and heat flow in metric measure spaces and applications to spaces with Ricci bounds from below}, Invent. Math. \textbf{195} (2014), no. 2, 289--391.

\bibitem{AGS} L.~Ambrosio, N.~Gigli, and G.~Savar\'e, \emph{Gradient flows in metric spaces and in the space of probability measures}, Second edition
Lectures Math. ETH Z\"urich, Birkh\"auser Verlag, Basel, 2008. x+334 pp.

\bibitem{AGSiber} L.~Ambrosio, N.~Gigli, and G.~Savar\'e, \emph{Density of Lipschitz functions and equivalence of weak gradients in metric measure spaces}, Rev. Mat. Iberoam. \textbf{29} (2013), no. 3, 969–-996.

\bibitem{Cheeger} J.~Cheeger, \emph{Differentiability of Lipschitz functions on metric measure spaces}, Geom. Funct. Anal., Vol. \textbf{6} (1999), 428--517

\bibitem{CG} J.~Czipszer, L.~Geh\'er, \emph{Extension of functions satisfying a Lipschitz condition}, Acta Math. Acad. Sci. Hungar. \textbf{6} (1955), 213--220.

\bibitem{DeBeVe} C.A.~De Bernardi, L.~Vesel\'y, \emph{On extension of uniformly continuous quasiconvex functions}, Proc. Amer. Math. Soc. \textbf{151} (2023), 1705--1716.

\bibitem{DGP} S.~Di Marino, N.~Gigli, and A.~Pratelli, \emph{Global Lipschitz extension preserving local constants}, Rend. Lincei Mat. Appl. \textbf{31} (2020), 757--765.

\bibitem{Gu25} V.~Gutev, \emph{On real-valued functions of Lipschitz type}, Expo. Math. \textbf{43} (2025), 125701.

\bibitem{KLR} J.~Koivu, D.~Lu\v ci\'c, and T.~Rajala, \emph{Approximation by BV-extension sets via perimeter minimization in metric spaces}, International Mathematics Research Notices \textbf{11} (2024), 9359--9375.

\bibitem{McShane}
E.J.~McShane, {\em Extension of range of functions}, Bull. Amer. Math.
  Soc. \textbf{40} (1934), 837--842.

\bibitem{Milman98}
V.A.~Mil\cprime man, {\em Lipschitz extensions of linearly bounded
  functions}, Mat. Sb. \textbf{189} (1998), 67--92.

\bibitem{W34}
H.~Whitney, {\em Analytic extensions of differentiable functions defined in closed sets}, Trans. Amer. Math. Soc. \textbf{36} (1934), 63--89.

\end{thebibliography}
\end{document}